\documentclass[11pt]{amsart}
\usepackage[foot]{amsaddr}
\usepackage[paper]{zbs}

\renewcommand{\vocab}{\emph}

\newcommand{\sfr}{\mathsf{r}}
\newcommand{\calk}{\mathcal{K}}
\newcommand{\cals}{\mathcal{S}}
\newcommand{\calw}{\mathcal{W}}

\newcommand{\thead}[2]{\multicolumn{1}{#1}{#2}}

\title{Sub-Fibonacci behavior in numerical semigroup enumeration}
\author{Daniel G.\ Zhu}
\address{Department of Mathematics, Massachusetts Institute of Technology, Cambridge, MA 02139}
\email{zhd@princeton.edu}
\begin{document}
\begin{abstract}
In 2013, Zhai proved that most numerical semigroups of a given genus have depth at most $3$ and that the number $n_g$ of numerical semigroups of a genus $g$ is asymptotic to $S\varphi^g$, where $S$ is some positive constant and $\varphi \approx 1.61803$ is the golden ratio. In this paper, we prove exponential upper and lower bounds on the factors that cause $n_g$ to deviate from a perfect exponential, including the number of semigroups with depth at least $4$. Among other applications, these results imply the sharpest known asymptotic bounds on $n_g$ and shed light on a conjecture by Bras-Amor\'os (2008) that $n_g \geq n_{g-1} + n_{g-2}$. Our main tools are the use of Kunz coordinates, introduced by Kunz (1987), and a result by Zhao (2011) bounding weighted graph homomorphisms.
\end{abstract}
\maketitle

\section{Introduction} \label{sec:intro}
A \vocab{numerical semigroup} $\Lambda$ is an additive submonoid of the nonnegative integers $\setn_0$ with finite complement $\setn_0\setminus \Lambda$. The \vocab{genus} $g(\Lambda)$ of a numerical semigroup is defined to be the size of $\setn_0 \setminus \Lambda$.

Over the past 15 years, significant progress has been made in understanding the number $n_g$ of numerical semigroups with genus $g$. Perhaps most significantly, in 2013 Zhai proved the following theorem, resolving two conjectures of Bras-Amorós \cite{BrasAmoros2008}:
\begin{thm}[Zhai \cite{Zhai2013}] \label{thm:zhaiphi}
We have $n_g \sim S\varphi^g$, where $S$ is some positive constant and $\varphi = \frac{1+\sqrt{5}}{2} \approx 1.61803$ is the golden ratio.
\end{thm}
However, the following conjecture remains unresolved:
\begin{conj}[Bras-Amorós \cite{BrasAmoros2008}] \label{conj:bafib}
For $g \geq 2$, $n_g \geq n_{g-1} + n_{g-2}$.
\end{conj}
In fact, it is not even known whether $n_g \geq n_{g-1}$ holds for all $g \geq 1$; while \cref{thm:zhaiphi} implies that $n_g \geq n_{g-1}$ for sufficiently large $g$, the bounds involved do not make it feasible to manually verify $n_g \geq n_{g-1}$ for small $g$.

Key properties of numerical semigroups also include the \vocab{multiplicity} $m(\Lambda) = \min(\Lambda \setminus \set{0})$, \vocab{conductor} $c(\Lambda) = \min \setmid{c \in \setn_0}{c + \setn_0 \subseteq \Lambda}$, and \vocab{Frobenius number} $f(\Lambda) = c(\Lambda) - 1$. Recently, Eliahou and Fromentin \cite{Eliahou2020} also introduced the \vocab{depth} $q(\Lambda) = \ceil{c(\Lambda)/m(\Lambda)}$. (From now on, we will omit the argument $\Lambda$ if it is clear from context.) For example, the numerical semigroup $\setn_0 \setminus \set{1, 2, 3, 4, 5, 7, 9, 10, 13}$ has genus $9$, multiplicity $6$, conductor $14$, Frobenius number $13$, and depth $3$.

Let $t_g$ be the number of numerical semigroups of genus $g$ satisfying $q \leq 3$ and let $\hat n_g = n_g - t_g$ be the number of numerical semigroups of genus $g$ with $q > 3$. This split is useful for enumerative purposes since ``most'' numerical semigroups satisfy $q \leq 3$, as proven by Zhai:
\begin{thm}[Zhai \cite{Zhai2013}] \label{thm:zhaimostthree}
We have $\lim_{g\to\infty} \frac{t_g}{n_g} = 1$. In particular, $t_g \sim S\varphi^g$.
\end{thm}
Recently, Eliahou and Fromentin proved the following bounds:
\begin{thm}[Eliahou-Fromentin \cite{Eliahou2020}] \label{thm:efbound}
For $g \geq 3$, $t_{g-1} + t_{g-2} \leq t_g \leq t_{g-1} + t_{g-2} + t_{g-3}$.
\end{thm}

In this paper, we study the asymptotics of the quantities $\hat n_g$ and $s_g = t_g - t_{g-1} - t_{g-2}$. Our main result is the following bound:
\begin{thm} \label{thm:mainbound}
Let $r_1 = \limsup_{g \to \infty} s_g^{1/g}$ and $r_2 = \limsup_{g \to \infty} \hat n_g^{1/g}$. Then $\sfr_{1.51} \leq r_1 \leq r_2 \leq \sfr_{1.54}$, where $\sfr_{1.51} \approx 1.51519$ and $\sfr_{1.54} \approx 1.54930$ are algebraic numbers defined in \cref{sec:prelim}.
\end{thm}
The two most important tools in the proof of \cref{thm:mainbound} are the use of Kunz coordinates and a result by Zhao giving bounds on graph homomorphisms.

One notable aspect of our proof is that it is independent of \cref{thm:zhaiphi,thm:zhaimostthree}. In fact, it is easy to see that \cref{thm:zhaimostthree} is a consequence of \cref{thm:mainbound}. Moreover, it is possible to use \cref{thm:mainbound} to prove the following bound on $n_g$:
\begin{thm} \label{thm:ngasymp}
Let $S = \frac{\varphi}{\sqrt{5}}(1 + \sum_{g \geq 3} s_g \varphi^{-g})$. Then $\limsup_{g \to \infty} \abs{n_g - S\varphi^g}^{1/g} \leq \sfr_{1.54}$.
\end{thm}
In particular, \cref{thm:ngasymp} implies that $\abs{n_g - S\varphi^g} = o(\varphi^g)$, so it is a strengthening of \cref{thm:zhaiphi} and provides an alternative expression for the constant $S$. (For a comparison of this expression with others, see \cref{rmk:scorr}.) Therefore, an additional contribution of this paper is to give alternative proofs of \cref{thm:zhaiphi,thm:zhaimostthree}.

We also have the following result, which may be seen as another generalization of \cref{thm:zhaimostthree}:
\begin{thm} \label{thm:twoeps}
Fix some $0 < \eps < 1$. Then, if $n_{g,\eps}$ is the number of numerical semigroups of genus $g$ with $f > (2+\eps)m$, we have $\limsup_{g \to \infty} n_{g,\eps}^{1/g} < \varphi$.
\end{thm}

More generally, our approach also provides an important framework for answering statistical questions about a ``typical'' numerical semigroup of a given genus. For example, we show in \cref{thm:fm2m} that, for most numerical semigroups $\Lambda$ of a given genus, the quantity $f(\Lambda) - 2m(\Lambda)$ is close to zero, and we determine its limiting distribution, strengthening results of Kaplan and Ye \cite{KaplanYe} and Singhal \cite{Singhal}. It is likely possible to extend these techniques to other statistics concerning numerical semigroups.

Finally, we develop and implement algorithms that evaluate $s_g$, and thus $t_g$, for all $g \leq 95$, a large improvement upon \cite{Eliahou2020}, which only computes $t_g$ for $g \leq 65$. Combining these results with known values for $n_g$ for $g \leq 72$ \cite[\href{https://oeis.org/A007323}{A007323}]{BrasAmoros2020,oeis}, we are able to make conjectures for the asymptotics of $\hat n_g$ and $s_g$ that open up an intriguing avenue for future study. Future work in this area could lead to a proof or disproof of \cref{conj:bafib} for sufficiently large $g$.

\subsection*{Comparison with the work of Bacher}
A recent preprint of Bacher \cite{bacher}, written independently of this work, reproves \cref{thm:zhaiphi,thm:zhaimostthree} in a manner similar to what we do here. Like this paper, Bacher focuses on Kunz coordinates and defines stressed words (though he calls them ``NSG-compositions with maximum $3$ ending with a part of maximal size''). Moreover, Bacher uses the same framework of bounding growth rates by looking at the convergence of generating functions, and proves both the lower bound $r_1 \geq \sfr_{1.51}$, in the language of \cref{thm:mainbound}, and the result that $S = \frac{\varphi}{\sqrt{5}}(1 + \sum_{g \geq 3} s_g \varphi^{-g})$. Moreover, he proposes the same algorithm for computing $s_g$ as we do in \cref{sec:program}, although he only computes $s_g$ up to $g=50$.

Our work improves upon that of Bacher in two main ways. For one, Bacher appears to be solely interested in proving that $r_1$ and $r_2$, in the language of \cref{thm:mainbound}, are strictly less than $\varphi$, and thus makes several ``inefficient'' arguments that simplify the analysis but mask the true growth rate. The second difference is that without Zhao's lemma bounding graph homomorphisms, Bacher bounds the relevant objects ``manually'', which becomes rather complicated even when the underlying graph structure is relatively simple; loosely speaking, in the language of \cref{sec:upper}, Bacher goes as far as the case $k = 3$, which already involves rather involved computations concerning crystallographic reflection groups and transfer matrices. This is enough to prove $r_1,r_2 < \varphi$, but it is unclear if this technique can be pushed further. Our work, in contrast, uses the language of graph homomorphisms to sidestep most of these complexities.

\subsection*{Outline}
\Cref{sec:prelim} introduces various definitions and lemmas used in the rest of the paper, while \cref{sec:kunz} develops the theory of Kunz coordinates and objects we call stressed words. The proof of \cref{thm:mainbound} is spread across \cref{sec:lower} and \cref{sec:upper}. \Cref{sec:statistics} discusses various applications of our results to the statistics of numerical semigroups, including the proofs of \cref{thm:ngasymp,thm:twoeps}. Finally, \cref{sec:program} discusses the computation of $s_g$, numerical results, and subsequent conjectures.

\section{Preliminaries} \label{sec:prelim}
\subsection{Notation and conventions}
The set of nonnegative integers is denoted $\setn_0$. For a positive integer $a$, let $[a] = \set{1, 2, \ldots, a}$. For real numbers $a < b$, it will also be useful to define the set $(a, b] = \setmid{n \in \setz}{n \geq 1, a < n \leq b}$. Observe that all elements of $(a, b]$ are positive integers.

The Fibonacci numbers are denoted $F_n$, with $F_0 = 0$, $F_1 = 1$, and $F_n = F_{n-1} + F_{n-2}$ for $n > 1$. There is also an explicit formula of $F_n = \frac{1}{\sqrt{5}}(\varphi^n - (-\varphi)^{-n})$.

We define the constants $\sfr_{1.51} \approx 1.51519$ and $\sfr_{1.54} \approx 1.54930$ so that $1/\sfr_{1.51}$ and $1/\sfr_{1.54}$ are the unique positive zeros of the polynomials $x^3(x+1)(x^2+x+1) - 1$ and $x^4(x+1)^2(x^4+2x^3+x^2+1) - 1$, respectively. Alternatively, $\sfr_{1.51}$ and $\sfr_{1.54}$ can be described as the unique positive zeros of the polynomials $x^6-x^3-2x^2-2x-1$ and $x^{10}-x^6-2x^5-2x^4-4x^3-6x^2-4x-1$.

\subsection{Bounding weighted graph homomorphisms}
In this section we state Zhao's graph homomorphism lemma and reframe it in a form which will be useful in \cref{sec:upper}.

All mentions of graphs in this paper are confined to this section. Here, all graphs are assumed to be finite, undirected, and with no multiple edges. However, they may contain loops unless otherwise stated. For any graph $G$, we let $V(G)$ denote the vertices of $G$ and $E(G)$ denote the edges of $G$. Let $K_{a,b}$ be the complete bipartite graph with parts of size $a$ and $b$.

Following \cite{Zhao2011}, we define a graph $H$, possibly with loops, to be a \vocab{threshold graph} if there is some function $g \colon V(H) \to \setr$ and real number $t$ such that for (possibly equal) $u,v \in V(H)$, we have $uv \in E(H)$ if and only if $g(u)+g(v) \leq t$.

For graphs $G$ and $H$, let $\Hom(G, H)$ be the set of graph homomorphisms\footnote{Recall that a graph homomorphism between two graphs $G$ and $H$ is a function $f \colon V(G) \to V(H)$ such that for every $uv \in E(G)$, we have $f(u)f(v) \in E(H)$.} from $G$ to $H$. Given an assignment of a nonnegative weight $z(v)$ to each vertex $v$ of $H$ and a homomorphism $f \in \Hom(G, H)$, let $z(f) = \prod_{v \in V(G)} z(f(v))$. Finally, let $\hom_z(G,H) = \sum_{f \in \Hom(G,H)} z(f)$. We can now state the following result:
\begin{thm}[{Zhao \cite[Cor.\ 7.6]{Zhao2011}}] \label{thm:zhao}
If $G$ is a loop-free, $k$-regular graph with $n$ vertices and $H$ is a threshold graph, then $\hom_z(G, H) \leq \hom_z(K_{k,k}, H)^{n/(2k)}$ for all nonnegative weight assignments $z$.
\end{thm}

Of use to us will be the following corollary:
\begin{cor} \label{cor:zhaocor}
Fix
\begin{itemize}
\item a finite set $V$ of real numbers;
\item a nonnegative weight $z(v)$ for every $v \in V$;
\item a real number $t$;
\item positive integers $d$ and $k$;
\item a subset $A \subseteq \setz/d\setz$ of size $k$, where $\setz/d\setz$ is the abelian group of integers modulo $d$.
\end{itemize}
Let $\mathcal{F}$ be the set of functions $f\colon \setz/d\setz \to V$ such that $f(a) + f(b) \geq t$ for all $a,b$ with $a + b \in A$. Furthermore, let $\mathcal{X}$ be the set of tuples $(a_1, \ldots, a_k, b_1, \ldots, b_k) \in V^{2k}$ with $\min(a_1, \ldots, a_k) + \min(b_1, \ldots, b_k) \geq t$. Then
\[\sum_{f \in \mathcal{F}} \prod_{a \in \setz/d\setz} z(f(a)) \leq \paren*{\sum_{(a_1, \ldots, a_k, b_1, \ldots, b_k) \in \mathcal{X}} \prod_{i=1}^k z(a_i)z(b_i)}^{d/(2k)}.\]
\end{cor}
\begin{proof}
Let $H$ be the graph with vertex set $V$ and an edge from $v$ to $v'$ if $v + v' \geq t$; this is clearly a threshold graph (set $g(v) = -v$). Now, construct a graph $G$ with vertex set $\setz/d\setz \times \set{1,2}$ by connecting an edge between $(a,i)$ and $(b,j)$, for $a,b\in \setz/d\setz$ and $i,j \in \set{1,2}$, if and only if both of the following conditions are satisfied:
\begin{enumerate}
\item  $a + b \in A$;
\item  $a \neq b$ and $i = j$, or $a = b$ and $i \neq j$.
\end{enumerate}
It is not hard to show that $G$ is $k$-regular, and moreover the second condition implies that $G$ is loop-free.

Applying \cref{thm:zhao}, we now find that $\hom_z(G, H) \leq \hom_z(K_{k,k}, H)^{d/k}$. Graph homomorphisms from $K_{k,k}$ to $H$ are given exactly by the elements of $\mathcal{X}$, and in fact
\[\hom_z(K_{k,k}, H) = \sum_{(a_1, \ldots, a_k, b_1, \ldots, b_k) \in \mathcal{X}} \prod_{i=1}^k z(a_i) z(b_i).\]
Therefore we only need to show that
\[\paren*{\sum_{f \in \mathcal{F}} \prod_{a \in \setz/d\setz} z(f(a))}^2 \leq \hom_z(G, H).\]
To do this, take arbitrary $f_1, f_2 \in \mathcal{F}$ and define the function $h \colon V(G) \to V$ by setting $h(a, i) = f_i(a)$. We claim that $h \in \Hom(G, H)$. To show this, note that if $a + b \in A$, $a \neq b$, and $i = j$, we have $h(a, i) + h(b, j) = f_i(a) + f_i(b) \geq t$. On the other hand, if $a + b \in A$, $a = b$, and $i \neq j$, then we have $h(a, i) + h(b, j) = \frac{1}{2}(f_i(a) + f_i(b) + f_j(a) + f_j(b)) \geq t$.

It is evident that the map $(f_1, f_2) \mapsto h$ is injective. Therefore
\[\sum_{f_1,f_2 \in \mathcal{F}} \prod_{a \in \setz/d\setz} z(f_1(a))z(f_2(a)) \leq \sum_{h \in \Hom(G, H)} \prod_{a \in \setz/d\setz} z(h(a, 1))z(h(a, 2)) = \hom_z(G, H),\]
which concludes the proof.
\end{proof}

\begin{rmk}
If we define $G'$ as the graph with vertex set $\setz/d\setz$ and vertices $a$ and $b$ connected if and only if $a + b \in A$, then \cref{cor:zhaocor} can be rephrased as $\hom_z(G', H) \leq \hom_z(K_{k,k}, H)^{d/(2k)}$, though this does not immediately follow from \cref{thm:zhao} since $G'$ can have loops. The trick used above to sidestep this issue can be generalized to prove \cref{thm:zhao} even when $G$ can have loops. (Here, a loop contributes $1$ to the degree of a vertex.)
\end{rmk}

\section{Kunz and Stressed Words} \label{sec:kunz}
\subsection{Kunz words}
Consider a numerical semigroup $\Lambda$ with multiplicity $m$. Since $m\in \Lambda$, by additive closure we have $\Lambda + m \subseteq \Lambda$. In particular, for each $0 \leq i < m$, there is some nonnegative integer $k_i$ such that
\[\setmid{n \in \Lambda}{n \equiv i \pmod{m}} = \set{k_i m + i, (k_i + 1)m + i, \ldots}.\]
Since $0 \in \Lambda$, we always have $k_0 = 0$. However, by the minimality of $m$, we have $k_i > 0$ for all other $i$. To capture this data, define the \vocab{Kunz coordinate vector} $\calk(\Lambda) = (k_1, k_2, \ldots, k_{m-1})$. These indices provide an important representation of $\Lambda$:
\begin{thm}[Kunz \cite{Kunz}, Rosales et al.\ \cite{Rosales2002}] \label{thm:kunzbijection}
The map $\calk$ is a bijection between numerical semigroups of multiplicity $m$ and $(m-1)$-tuples of positive integers $(k_1, k_2, \ldots, k_{m-1})$ such that $k_i + k_j \geq k_{i+j}$ and $k_i + k_j + 1 \geq k_{i+j-m}$ for all indices $i$, $j$ for which these inequalities are defined.
\end{thm}
A common method by which to interpret this result is to imagine these tuples as the integral points inside a convex polyhedron in $(m-1)$-dimensional space. However, in this paper we will make something of a psychological shift to instead consider Kunz coordinate vectors as words drawn from an alphabet of positive integers.
\begin{defn}
A (possibly empty) word $w = w_1w_2\cdots w_\ell$ of positive integers is \vocab{Kunz} if $w_i + w_j \geq w_{i+j}$ and $w_i + w_j + 1 \geq w_{i+j-\ell-1}$ for all $i, j$ for which these inequalities are defined. If we additionally have $w_i \leq q$ for all $i$, the word $w$ is said to be \vocab{$q$-Kunz}.
\end{defn}
Thus, $\calk$ is a bijection between numerical semigroups and Kunz words.
\begin{examp}
The Kunz word of $\setn_0$ is the empty word. If $\Lambda = \setn_0 \setminus \set{1, 2, 3, 4, 5, 7, 9, 10, 13}$, then $\calk(\Lambda) = 31221$.
\end{examp}
Given this bijection, one can define the multiplicity, genus, depth, conductor, and Frobenius number of a Kunz word as that of its corresponding numerical semigroup. These properties can also be described more naturally in terms of the word itself:
\begin{prop} \label{prop:kunzprops}
Let $\Lambda$ be a numerical semigroup and let $w = w_1w_2\cdots w_\ell = \calk(\Lambda)$ be its Kunz word. Then
\begin{itemize}
\item $m(w) = \ell + 1$;
\item $g(w) = \sum_{i \in [\ell]} w_i$;
\item $q(w) = \max_{i \in [\ell]} w_i$, or $0$ if $w$ is empty.
\item $f(w) = (q - 1)m + j$, where $j$ is maximal such that $w_j = q$, or $j = 0$ if $w$ is empty.
\end{itemize}
\end{prop}
\begin{proof}
If $\Lambda = \setn_0$, which corresponds to $w$ being empty, we have $m = 1$, $g = 0$, $q = 0$ and $f = -1$, which is consistent with the expressions for $m, g, q, f$. Henceforth assume $\Lambda \neq \setn_0$.

The fact that $m = \ell + 1$ follows from the definition of $\calk$. Also, from the definition of $\calk$ we find
\[\setn_0 \setminus \Lambda = \bigcup_{i \in [\ell]} \set{i, m + i, \ldots, (w_i - 1)m + i}.\]
Thus, $g = \abs{\setn_0 \setminus \Lambda} = \sum_{i \in [\ell]} w_i$ and $f = \max(\setn_0 \setminus \Lambda) = (\max_{i \in [\ell]} w_i - 1)m + j$. Finally, $q = \ceil{c/m} = \max_{i \in [\ell]} w_i - 1 + \ceil*{\frac{j+1}{m}} = \max_{i \in [\ell]} w_i$.
\end{proof}
In particular, $q$-Kunz words of genus $g$ correspond to semigroups of genus $g$ and depth at most $q$.

\subsection{\texorpdfstring{$2$}{2}-Kunz words}
As a warmup, we will first demonstrate how Kunz words can enumerate semigroups of depth at most $2$.

Consider a word $w = w_1w_2 \cdots w_\ell$ on the alphabet $\set{1,2}$.  For all $i, j$, we find $w_i + w_j \geq 2 \geq w_{i+j}$ and $w_i + w_j + 1 \geq 3 \geq w_{i+j-\ell-1}$, so $w$ is automatically $2$-Kunz.

\begin{prop}[\cite{Zhao2010}] \label{prop:2kunz}
The number of numerical semigroups of genus $g$ and depth at most $2$ is $F_{g+1}$.
\end{prop}
\begin{proof}
Let the number of $2$-Kunz words of genus $g$ be $a_g$. We can calculate $a_0 = 1 = F_1$ and $a_1 = 1 = F_2$. Moreover, every nonempty $2$-Kunz word can be generated from a shorter $2$-Kunz word by appending a $1$ or a $2$, so for $g > 1$ we find $a_g = a_{g-1} + a_{g-2}$. By induction, $a_g = F_{g+1}$. 
\end{proof}

\subsection{\texorpdfstring{$3$}{3}-Kunz words}
Consider a word $w = w_1w_2 \cdots w_\ell$ on the alphabet $\set{1,2,3}$. For all $i, j$, we find $w_i + w_j + 1 \geq 3 \geq w_{i+j-\ell-1}$, so $w$ is $3$-Kunz if and only if there do not exist indices $i$ and $j$ with $w_i = w_j = 1$ and $w_{i+j} = 3$.
\begin{prop} \label{prop:3kunzop}
\begin{parts}
\item If $w$ is nonempty and $3$-Kunz, then the word $w'$ formed by deleting the last character in $w$ is also $3$-Kunz.
\item If $w$ is $3$-Kunz, so are the words $w1$ and $w2$. (Here $w1$ refers to the word created by appending a $1$ to the end of $w$, and similarly for $w2$.)
\end{parts}
\end{prop}
\begin{proof}
If $w'$ is not $3$-Kunz, then there are $i, j$ with $w'_i = w'_j = 1$ and $w'_{i+j} = 3$. But then $w_i = w_j = 1$ and $w_{i+j} = 3$, contradicting the fact that $w$ is $3$-Kunz.

Now suppose $v = w1$ is not $3$-Kunz; call its length $\ell$. If $v_i = v_j = 1$ and $v_{i+j} = 3$, then since $v_\ell = 1$ we find $i, j, i+j < \ell$. Thus $w_i = w_j = 1$ and $w_{i+j} = 3$, again contradicting the fact that $w$ is $3$-Kunz. Similar reasoning proves that $w2$ is $3$-Kunz.
\end{proof}

At this point, we will now give a short proof of \cref{thm:efbound}.
\begin{proof}[Proof of \cref{thm:efbound}]
The number $t_g$ counts the number of $3$-Kunz words of genus $g$. By \cref{prop:3kunzop}, those that end with $1$ are precisely the words $w1$ where $w$ is any $3$-Kunz word of genus $g-1$, of which there are $t_{g-1}$. Similarly, there are exactly $t_{g-2}$ such words that end in $2$. The words that end in $3$ must all be of the form $w3$ for some $3$-Kunz word $w$ of genus $g-3$, but not all words of this form may necessarily be $3$-Kunz. Hence, $t_{g-1} + t_{g-2} \leq t_g \leq t_{g-1} + t_{g-2} + t_{g-3}$, as desired.
\end{proof}

This proof motivates the definition of a stressed word.
\begin{defn}
A word of positive integers is \vocab{stressed} if it is $3$-Kunz and ends in $3$.
\end{defn}
It is apparent from the above proof that the number of stressed words of genus $g \geq 3$ is $t_g - t_{g-1} - t_{g-2}$, a quantity we will define to be $s_g$. To be consistent with the idea that $s_g$ counts stressed words of a given genus, we also define $s_0 = s_1 = s_2 = 0$.
\begin{examp}
The stressed words of length at most $3$ are $3$, $23$, $33$, $123$, $213$, $313$, $223$, $233$, $323$, and $333$.
\end{examp}

\begin{prop} \label{prop:3kdecomp}
Kunz words of depth $3$ are in bijection with pairs of stressed words and $2$-Kunz words, with the bijection given by concatenation.
\end{prop}
\begin{proof}
Suppose $w = w_1w_2\cdots w_\ell$ is Kunz of depth $3$. If $i$ is maximal such that $w_i = 3$, then \cref{prop:3kunzop} implies $w_1w_2\cdots w_i$ is stressed and $w_{i+1}w_{i+2}\cdots w_\ell$ is $2$-Kunz. Moreover, if $w$ is stressed and $v$ is $2$-Kunz, then \cref{prop:3kunzop} yields that the concatenation of $w$ and $v$ is Kunz of depth $3$. These two operations are clearly inverse to each other.
\end{proof}
\begin{cor} \label{cor:tgdecomp}
For all $g \geq 0$, $t_g = F_{g+1} + \sum_{i=1}^g s_i F_{g+1-i}$.
\end{cor}

\begin{defn} \label{def:prefix}
Let $w$ be a $3$-Kunz word. If $q(w) = 3$, \cref{prop:3kdecomp} associates a stressed word with $w$, which we call the \emph{prefix} of $w$. If $q(w) \leq 2$, we define the prefix of $w$ to be the empty word.
\end{defn}

\section{First Bounds on Growth Rates} \label{sec:lower}
In this section, we prove the first two bounds of \cref{thm:mainbound}, namely that $r_1 \geq \sfr_{1.51}$ and $r_1 \leq r_2$, deferring the proof that $r_2 \leq \sfr_{1.54}$, by far the most involved part, to \cref{sec:upper}.

\subsection{A lower bound on \texorpdfstring{$r_1$}{r₁}}
Since $w$ is $3$-Kunz if and only if there are no $i, j$ with $w_i = w_j = 1$ and $w_{i+j} = 3$, a sufficient condition for a word of length $\ell$ to be $3$-Kunz is that $w_i > 1$ for all $i \leq \ell/2$. In particular, if the length is $2k + 1$ for $k \geq 0$, a word is stressed if the first $k$ characters are in the set $\set{2,3}$, the next $k$ are in the set $\set{1,2,3}$, and the last is $3$. Therefore, if we define the generating function
\[f(x) = \sum_{k \geq 0} x^3(x^2+x^3)^k(x+x^2+x^3)^k = \frac{x^3}{1-x^3(x+1)(x^2+x+1)} = \sum_{g \geq 0} s'_g x^g,\]
we find that $s'_g \leq s_g$.

We now claim that $\lim_{g \to \infty} (s'_g)^{1/g} = \sfr_{1.51}$. This is a standard result in the theory of generating functions; a more detailed exposition can be found in many standard texts, e.g.\ \cite[Ch.\ 5]{wilf}. In our case, observe that $1-x^3(x+1)(x^2+x+1)$ has a unique zero of smallest magnitude, namely $1/\sfr_{1.51}$, which happens to be simple. Therefore, by the theory of partial fractions, $f(x)$ can be written as $\tilde f(x) + \frac{C}{x - (1/\sfr_{1.51})}$, where $\tilde f(x)$ has a radius of convergence strictly greater than $1/\sfr_{1.51}$ and $C$ is some nonzero real constant. The result follows immediately.

\subsection{Proof that \texorpdfstring{$r_1 \leq r_2$}{r₁ ≤ r₂}}
We first prove a useful lemma.
\begin{lem}
If $g \geq 7$, then $s_g \leq 2\hat n_{g-2} + 3\hat n_{g-3} + 2\hat n_{g-4} + \hat n_{g-5}$.
\end{lem}
\begin{proof}
If $w = w_1w_2\cdots w_{\ell-1}3$ is stressed of genus $g$ and has length $\ell \geq 3$, let $\Phi(w) = 4w_2w_3\cdots w_{\ell - 2}$. We claim that $\Phi(w)$ is Kunz. Indeed, for $i, j \in [\ell - 2]$ with $i + j \leq \ell-2$, we have $\Phi(w)_i + \Phi(w)_j \geq w_i + w_j \geq w_{i+j} = \Phi(w)_{i+j}$. Also, if $i + j \geq \ell+1$, then $\Phi(w)_i + \Phi(w)_j + 1\geq 3 \geq \Phi(w)_{i+j-(\ell-1)}$. Lastly, if $i + j = \ell$, then $\Phi(w)_i + \Phi(w)_j \geq w_i + w_j \geq w_{i+j} = 3 = \Phi(w)_1 - 1$.

For $g$ a positive integer and $a, b \in [3]$, let $\cals(g, a, b)$ be the set of stressed words of genus $g$ and of the form $w = aw_2\cdots w_{\ell-2}b3$, where $\ell \geq 3$ is arbitrary. Since every stressed word of genus $g \geq 7$ has length at least $3$, and $\cals(g, 1, 1)$ is empty, we know that for $g \geq 7$ we have
\[s_g = \sum_{\substack{a,b\in [3] \\ (a,b) \neq (1,1)}} \abs{\cals(g, a, b)}.\]
On the other hand, $\Phi$, when restricted to each of the $\cals(g, a, b)$, is injective, and produces a Kunz word of genus $g + 1 - a - b$ and depth $4$. Therefore $\abs{\cals(g, a, b)} \leq \hat n_{g+1-a-b}$. The result follows.
\end{proof}

To finish, we note that by the definition of $r_2$, for all $\eps > 0$ there is some $g_0$ such that $\hat n_g < (r_2 + \eps)^g$ for all $g > g_0$. (If $r_2$ were infinite the statement that $r_1 \leq r_2$ would be tautological.) Therefore, for every $\eps > 0$ we find that for sufficiently large $g$, 
\[s_g < 2(r_2+\eps)^{g-2}+3(r_2+\eps)^{g-3}+2(r_2+\eps)^{g-4}+(r_2+\eps)^{g-5} < 8(r_2+\eps)^g.\]
Thus $r_1 \leq r_2 + \eps$ for all $\eps$, implying that $r_1 \leq r_2$.

\section{The Main Upper Bound} \label{sec:upper}
\subsection{Introduction and outline of proof}
In this section we will prove that $r_2 \leq \sfr_{1.54}$, finishing the proof of \cref{thm:mainbound}. Our general technique will be to use the fact that $1/r_2$ is the radius of convergence of the generating function $\sum_{g \geq 1} \hat n_g x^g$. In particular, we will be done if we can show that if $\abs{x} < \sfr\inv_{1.54}$, then the sum $\sum_{g \geq 1} \hat n_g x^g$ converges. Since $\hat n_g \geq 0$ for all $g$, we may additionally assume that $x$ is a positive real number. It will be further convenient to assume that $x > 5/8$.

At this point, we let $k$ be a positive integer that we will determine at a later time. It will be true that as $x$ approaches $\sfr\inv_{1.54}$, the value of $k$ needed for the proof to work will become arbitrarily large.

Define the \vocab{$k$-dense depth} $q'$ of a Kunz word to be the largest positive integer such that there are at least $k^2$ occurrences of $q'$ in the word, or $0$ if such a $q'$ does not exist. Note that $q'$ is at most the depth $q$. For positive integers $q, q', \ell, p$ with $q' \leq q$ and $p \leq \ell$, define the following sets:
\begin{align*}
\cals_q(\ell) &= \set{\text{Kunz words of depth $q$ and length $\ell$}}, \\
\cals_{q,q'}(\ell) &= \set{\text{Kunz words of depth $q$, $k$-dense depth $q'$, and length $\ell$}},\\
\cals_{q, q'} (\ell, p) &= \set{\text{words } w \in \cals_{q, q'}(\ell) \text{ such that }w_p = q}.
\end{align*}
Moreover, given a word $w = w_1w_2\cdots w_\ell$ consisting of positive integers, we define its weight $z(w) = x^{\sum_{i\in [\ell]} w_i}$. In this section, we will work with several (finite) sets of such words with names containing the symbol $\cals$, like those defined above. As a notational shorthand, we denote the sum of the weights of the elements of such sets by switching the $\cals$ into a $\calw$. For example,
\[\calw_q(\ell) = \sum_{w \in \cals_q(\ell)} z(w).\]

The main structure of the proof is to break up the sum $\sum_{g \geq 1} \hat n_g x^g = \sum_{\ell \geq 1} \sum_{q \geq 4} \calw_q(\ell)$ into five cases, as follows:
\begin{enumerate}[label=\Roman*.]
\item All words with $q \geq 9$.
\item Words with $q' \leq 1$ and $4 \leq q \leq 8$.
\item Words with $q' = 2$ and $4 \leq q \leq 8$.
\item Words with $q' = 3$ and $4 \leq q \leq 8$.
\item Words with $4 \leq q' \leq q \leq 8$.
\end{enumerate}
Specifically, Case I establishes that $\sum_{\ell \geq 1} \sum_{q \geq 9} \calw_q(\ell)$ converges, after which it will suffice to show that for the finitely many remaining pairs $(q, q')$, the sum $\sum_{\ell \geq 1} \calw_{q,q'}(\ell)$ converges, which is the content of the remaining cases. For each $(q, q')$, we will show that the sum converges for sufficiently large $k$, meaning that by making $k$ large enough, all of these sums will converge, which will show the desired result.

For Cases III, IV, and V, the general strategy will be to use the pigeonhole principle to find occurrences of $q'$ that are close together in the word, and then to apply a ``standardization map'' to transform the word into objects that can either be handled directly, as in Case III, or bounded using \cref{cor:zhaocor}, as in Cases IV and V. These techniques will be developed in two interlude subsections.

The only part of the proof that is critically dependent on the condition $x < \sfr_{1.54}\inv$ is Case IV. The remainder of the proof can be easily modified to work with any $x < \sfr_{1.51}\inv$.

\subsection{Case I: \texorpdfstring{$q \geq 9$}{q ≥ 9}}
The objective of this section is to find a $q_0$ such that we can prove that $\sum_{\ell = 1}^\infty \sum_{q = q_0}^\infty \calw_q(\ell)$ converges. To see this, pick some $p \in [\ell]$ and consider all Kunz words $w = w_1w_2\cdots w_\ell$ such that $w_p = q \geq q_0$. Since $w$ is Kunz, it can be easily shown that $w_i + w_j \geq w_p -1 \geq q_0 - 1$ whenever $i + j \equiv p \pmod{\ell + 1}$. The map $i \mapsto p - i \pmod{\ell + 1}$ is an involution on $[\ell] \setminus \set{p}$, so we conclude that the sum of the weights of all Kunz words of this form is at most $x^q\alpha^a\beta^b$
where $a$ and $b$ are the number of cycles of length $1$ and $2$ created by this involution, respectively, and
\[\alpha = \sum_{2i \geq q_0 - 1} x^i \text{ and } \beta = \sum_{i + j \geq q_0 - 1} x^{i+j}.\]
Since $\alpha^2 \leq \beta$, this is bounded above by $x^q \beta^{\frac a2+b} = x^q \beta^{\frac{\ell - 1}{2}}$. Therefore, summing over all $q$ and $p$,
\[\sum_{q = q_0}^\infty \calw_q(\ell) \leq \ell \paren*{\sum_{q \geq q_0} x^q}\beta^{\frac{\ell - 1}{2}}.\]
This expression converges when summed over all $\ell$ if and only if $\beta < 1$. Computations with software show that $q_0 = 9$ satisfies this condition.

\subsection{Case II: \texorpdfstring{$q' \leq 1$}{q' ≤ 1}}
Fix some $q' \leq 1$ and $4 \leq q \leq 8$ and consider some $w \in \cals_{q,q'}(\ell, p)$.

Again, it is easy to see that $w_i + w_j \geq w_p-1 =  q - 1$ for all $i + j \equiv p \pmod{\ell + 1}$, so by summing over all $i \neq p$ we find that $2\sum_{i \neq p} w_i \geq (q-1)(\ell - 1)$, meaning that the average of all the $w_i$ is greater than $\frac{q-1}{2} \geq \frac{3}{2} > q'$.

On the other hand, there are at most $(q - q')k^2$ integers larger than $q'$ in any word in $\cals_{q,q'}$, so the average of all the $w_i$ is bounded above by
\[\frac{(q-q')k^2 \cdot q + (\ell - (q-q')k^2) \cdot q'}{\ell} = q' + \frac{(q-q')^2k^2}{\ell}.\]

Therefore, $\cals_{q,q'}(\ell)$ is empty for sufficiently large $\ell$. Therefore $\sum_{\ell = 1}^\infty \calw_{q,q'}(\ell)$ converges\footnote{This same argument works whenever $q>2q'+1$, but using this fact for $q' \geq 2$ is unnecessary for our purposes as we can get the high-$q$ cases ``for free'' in our later arguments.}.

\subsection{Interlude I: Pigeonhole trick}
For $d, p \in [\ell]$, let
\begin{align*}
\cals_{q, q'} (\ell, d, p) &= \set{\text{words } w \in \cals_{q, q'}(\ell, p) \text{ with at least $k$ values of $i \in (d - \ell/k, d]$ with $w_i = q'$}}, \\
\cals'_{q,q'}(\ell, d,p) &= \set{\text{words } w \in \cals_{q,q'}(\ell, d,p) \text{ such that } s \notin \cals_{q,q'}(\ell, d',p)\text{ for all } d < d' < \ell}.
\end{align*}
\begin{lem} \label{lem:pigeonhole}
If $q' > 0$, then
\[\cals_{q,q'}(\ell) = \bigcup_{d=1}^\ell \bigcup_{p=1}^\ell \cals'_{q, q'}(\ell, d, p).\]
In particular,
\[\calw_{q,q'}(\ell) \leq \sum_{d=1}^\ell \sum_{p=1}^\ell \calw'_{q, q'}(\ell, d, p).\]
\end{lem}
\begin{proof}
Obviously, $\cals_{q,q'}(\ell) = \bigcup_{p=1}^\ell \cals_{q,q'}(\ell, p)$. For $w \in \cals_{q,q'}(\ell, p)$, by the definition of $q'$ there are at least $k^2$ values of $i$ with $w_i = q'$. The $k$ intervals $(0, \ell/k], (\ell/k, 2\ell/k], \ldots, ((k-1)\ell/k, \ell]$ cover $[\ell]$, so one of them must have at least $k$ such values of $i$. Therefore, there is some $d$ such that $w \in \cals_{q,q'}(\ell, d, p)$. By taking $d$ to be maximal, we find that $w \in \cals'_{q,q'}(\ell, d, p)$.
\end{proof}
\begin{prop} \label{prop:pigeonholeprop}
If $w \in \cals'_{q,q'}(\ell, d, p)$, there are at most $k^2$ values of $i > d$ with $w_i = q'$.
\end{prop}
\begin{proof}
If not, then ignore all the $i \leq d$ with $w_i = q'$ and apply the same argument in \cref{lem:pigeonhole} to find some $d'$ such that there are at least $k$ values of $i \in (d'-\ell/k, d']$ and $i > d$ with $w_i = q'$. Such a $d'$ must be greater than $d$, and we also have $w \in \cals_{q,q'}(\ell, d', p)$, which is a contradiction.
\end{proof}

\subsection{Interlude II: Standardization}
Now, we aim to convert the set $\cals'_{q,q'}(\ell, d, p)$ into a form that is easier to work with by jettisoning most of the ``extraneous material'', like all the occurrences of numbers greater than $q'$.  To do this, we first define the following operation:
\begin{defn}
Given $w \in \cals'_{q,q'}(\ell, d, p)$ let $I \subseteq [\ell]$ be the set of indices $i$ such that at least one of the following holds:
\begin{itemize}
\item $i \in (d - \ell/k, d]$,
\item $i \in [d]$ and $w_i > q'$, or
\item $i > d$ and $w_i \geq q'$.
\end{itemize}
Then, define $w^\circ = w^\circ_1w^\circ_2\cdots w^\circ_\ell$ so that
\begin{itemize}
\item if $i \in I$, then $w^\circ_i = q'$ if $i \leq d$ and $w^\circ_i = \max(q' - 1, 2)$ if $i > d$;
\item otherwise, $w^\circ_i = w_i$.
\end{itemize}
\end{defn}
\begin{defn}
If $A$ is a subset of $(d-\ell/k, d]$ of size $k$, let $\cals^\circ_{q,q'}(\ell, d, p, A)$ be the image of this map $w \mapsto w^\circ$ over all $w \in \cals'_{q,q'}(\ell, d, p)$ such that $w_i = q'$ for all $i \in A$.
\end{defn}

\begin{prop} \label{prop:simpprop}
A word $w \in \cals^\circ_{q,q'}(\ell, d, p, A)$ satisfies the following conditions:
\begin{itemize}
\item $w_i \in [q']$ for all $i \in [\ell]$ and $w_i \leq \max(q' - 1, 2)$ for all $i > d$,
\item $w_i + w_j \geq q'$ for all $i, j \in [d]$ with $i + j \in A \pmod{d}$ (where we treat $[d]$ as $\setz/d\setz$),
\item $w_i + w_j \geq 3$ for all $i, j \in [\ell]$ such that $i + j \equiv p \pmod{\ell + 1}$.
\end{itemize}
\end{prop}
\begin{proof}
Let $w \in \cals'_{q,q'}(\ell, d, p)$ be such that $w_i = q'$ for all $i \in A$. It suffices to check the conditions for $w^\circ$.

The first condition is easy.  For the second, note that if $i \in I$ or $j \in I$ then the condition is true. Otherwise, we know that $i, j \leq d - \ell/k$, so $i + j \leq 2d - \ell/k$, implying that $i + j \in A$. Therefore $w^\circ_i + w^\circ_j = w_i + w_j \geq w_{i+j} = q'$.

To check the third condition, again note that if $i \in I$ or $j \in I$ then the condition is true. Otherwise, if $i + j = p$, then $w^\circ_i + w^\circ_j = w_i + w_j \geq w_{i+j} = q \geq 4$. If $i + j = p + \ell + 1$, then $w^\circ_i + w^\circ_j = w_i + w_j \geq w_{i+j-\ell-1} - 1  = q-1 \geq 3$.
\end{proof}
The map $w \mapsto w^\circ$ does not lose too much information:
\begin{lem} \label{lem:redweight}
If $q,q',\ell,d,p,A$ are as defined above,
\[\calw'_{q,q'}(\ell, d, p) \leq \ell^{10k^2} 1000^{10k^2 + \ceil{\ell/k}}\sum_A \calw^\circ_{q,q'}(\ell, d, p, A).\]
\end{lem}
\begin{proof}
Every word $w \in \cals'_{q,q'}(\ell, d, p)$ has some subset $A \subseteq (d-\ell/k, d]$ of size $k$ such that $w_i = q'$ for all $i \in A$. (In fact there may be multiple such $A$.) Thus, $w^\circ \in \cals^\circ_{q,q'}(\ell,d,p,A)$ for some $A$. We now claim that for any word $w'$ in the range of $w \mapsto w^\circ$,
\[\sum_{\substack{w\in \cals'_{q,q'}(\ell,d,p) \\ w^\circ = w'}} z(w) \leq \ell^{10k^2}1000^{10k^2+\ceil{\ell/k}} z(w').\] 
Then, by summing over all $w'$, it will follow that
\begin{align*}
\calw'_{q,q'}(\ell,d,p) &\leq \ell^{10k^2}1000^{10k^2+\ceil{\ell/k}} \sum_{w' \in \bigcup_A \cals^\circ_{q,q'}(\ell,d,p,A)} z(w') \\
&\leq \ell^{10k^2} 1000^{10k^2 + \ceil{\ell/k}}\sum_A \calw^\circ_{q,q'}(\ell, d, p, A).
\end{align*}

Given some $w' = w^\circ$, the word $w$ can be reconstructed using only the information of $I$ and a choice of $w_i \in [q]$ for all $i \in I$. The set $I$ can be decomposed as $(d-\ell/k, d] \cup I'$, where $I'$ is the set of indices $i$ such that $i \leq d-\ell/k$ and $w_i > q'$, or $i > d$ and $w_i \geq q'$. If $q' < a \leq q$, then by the definition of $q'$ there exist at most $k^2$ values of $i$ with $w_i = a$. Moreover, by \cref{prop:pigeonholeprop} there are at most $k^2$ values of $i \in I'$ with $w_i = q'$. Therefore $\abs{I'} \leq (q-q'+1)k^2 \leq 10k^2$. Since $d \notin I'$, we conclude that the number of choices for $I$, which is entirely determined by $I'$, is at most\footnote{We are using the fact that for positive integers $a$ and $b$, $\binom{a-1}{0} + \cdots + \binom{a-1}{b} \leq 1 + (a-1) + \cdots + (a-1)^b \leq a^b$.}
\[\binom{\ell-1}{0} + \binom{\ell-1}{1} + \cdots + \binom{\ell-1}{10k^2} \leq \ell^{10k^2}.\]
Moreover, this bound on $\abs{I'}$ implies that $\abs{I} \leq 10k^2+\ceil{\ell/k}$. Now, we fix $I$, the sum of the weights of the words produced by, for every $i \in I$, replacing $w_i'$ with an arbitrary element of $[q]$ is
\[z(w')\prod_{i \in I}\frac{x + x^2 + \cdots + x^q}{x^{w'_i}} \leq z(w')(qx^{1-q'})^{\abs{I}} \leq z(w')1000^{10k^2+\ceil{\ell/k}},\]
where we have used the fact that $qx^{1-q'} \leq 8\cdot (8/5)^7 \leq 1000$. Summing over all $I$ proves the bound.
\end{proof}

We make one last definition in this section:
\begin{defn}
For $i, j \in [\ell]$, we say $i \sim j$ if and only if $i \neq j$ and $i + j \equiv p \pmod{\ell + 1}$.
\end{defn}
The relation $\sim$ breaks $[\ell]$ into some number of pairs, with the only elements not in a pair being $p$ or any element $i$ with $2i \equiv p \pmod{\ell+1}$, of which there are at most $3$ in total.

\subsection{Case III: \texorpdfstring{$q' = 2$}{q' = 2}}
Let $S$ be the set of indices $i$ such that $i \sim j$ for some $j$. For all $w \in \cals^\circ_{q,2}(\ell,d,p,A)$, we have that $w_i \leq 2$ for all $i$ and $w_i + w_j \geq 3$ for all $i, j$ with $i \sim j$. Therefore
\[\calw^\circ_{q,2}(\ell, d, p, A) = (2x^3+x^4)^{\abs{S}/2} (x + x^2)^{\ell - \abs{S}}.\]
Since $\ell - \abs{S} \leq 3$, and $\sqrt{2\sfr_{1.54}^{-3} + \sfr_{1.54}^{-4}} \leq 0.9$, we find that
\[\calw^\circ_{q,2}(\ell, d, p, A) \leq 0.9^{\ell - 3} 2^3 \leq 100 \cdot 0.9^\ell.\]
Therefore, after applying \cref{lem:redweight} and summing over at most $\ell^k$ possible values for $A$, $\ell$ possible values for $d$, and $\ell$ possible values for $p$, we get
\[\calw_{q,2}(\ell) \leq 100 \cdot \ell^{10k^2+k+2} 1000^{10k^2 + \ceil{\ell/k}} 0.9^\ell.\]
Summing over all $\ell$, we find that this converges if and only if $1000^{1/k} \cdot 0.9 < 1$, which is true for sufficiently large $k$.

\subsection{Case IV: \texorpdfstring{$q' = 3$}{q' = 3}} \label{subsec:caseiv}
Under the relation $\sim$, we may decompose the set $[\ell]$ into five sets:
\begin{itemize}
\item $S_1$, the set of elements $i \leq d$ not paired with an element greater than $d$,
\item $S_2$, the set of elements $i \leq d$ paired with an element greater than $d$,
\item $S_3$, the set of elements $i > d$ paired with an element in $[d]$,
\item $S_4$, the set of elements $i > d$ paired with an element greater than $d$,
\item $S_5$, the set of elements $i > d$ not paired with any element.
\end{itemize}
Now, every element $w \in \cals_{q,3}^\circ(\ell, d, p, A)$ can be generated via the following process:
\begin{itemize}
\item Choose $w_1, w_2,\ldots, w_d \in [3]$ so that $w_i + w_j \geq 3$ for all $i, j$ with $i + j \in A \pmod{d}$.
\item For all $i \in S_2$ and $j \in S_3$ with $i \sim j$, choose $w_j \in [2]$ so that $w_i + w_j \geq 3$.
\item For all $i, j \in S_4$ with $i \sim j$, choose $w_i, w_j \in [2]$ so that $w_i + w_j \geq 3$.
\item For all $i \in S_5$, choose $w_i \in [3]$ arbitrarily.
\end{itemize}
Now, if we let
\[u(a) = \sum_{\substack{b \in [2] \\a+b \geq 3}} x^{a+b}\quad \text{and} \quad u'(a) = x^a,\] we find that the weighted sum of all the words produced this way is
\[\paren*{\sum_{w_1,w_2,\ldots,w_d} \prod_{i \in S_1} u'(w_i) \prod_{i \in S_2} u(w_i)} (2x^3+x^4)^{\abs{S_4}/2} (x + x^2 + x^3)^{\abs{S_5}},\]
where the sum is over all $w_1,w_2,\ldots,w_d \in [3]$ that satisfy $w_i + w_j \geq 3$ for all $i, j$ with $i + j \in A \pmod{d}$. As a result, since $(2x^3+x^4)^{1/2}\leq 0.9$ and $(x+x^2+x^3)^{\abs{S_5}} \leq 1.5^{\abs{S_5}} \leq 10$, we conclude that
\[\calw^\circ_{q,3}(\ell, d, p, A) \leq \paren*{\sum_{w_1,w_2,\ldots,w_d} \prod_{i \in [d]} u''(w_i)} \cdot 10 \cdot 0.9^{\abs{S_4}},\]
where we let $u''(a) = \max(u(a), u'(a))$. The first term may be bounded using \cref{cor:zhaocor}\footnote{We apply \cref{cor:zhaocor} with the parameters $V 
 =[3]$, $z(v) = u''(v)$, and $t = 3$. The parameters $d$, $k$, and $A$ are the same.}, which yields that it is at most
\[\paren*{\sum_{(a_1, \ldots, a_k, b_1, \ldots, b_k) \in \mathcal{X}} \prod_{i=1}^k u''(a_i)u''(b_i)}^{d/(2k)},\]
where $\mathcal{X}$ is the set of tuples $(a_1,\ldots,a_k,b_1,\ldots,b_k)\in [3]^{2k}$ that satisfy $\min(a_1,\ldots,a_k) + \min(b_1,\ldots,b_k) \geq 3$. We may decompose $\mathcal{X} = (\set{2,3}^k \times [3]^k) \cup ([3]^k \times \set{2,3}^k)$, so
\[
\sum_{(a_1, \ldots, a_k, b_1, \ldots, b_k) \in \mathcal{X}} \prod_{i=1}^k u''(a_i)u''(b_i) \leq 2(u''(1)+u''(2)+u''(3))^k(u''(2)+u''(3))^k.
\]
Since $x > 5/8$, we have $x + x^2 > 1$, from which we find $u''(1) = \max(x^3,x) = x$, $u''(2) = \max(x^2(x+x^2),x) = x^3(x+1)$, and $u''(3) = \max(x^3(x+x^2),x^3) = x^4(x+1)$. So
\[(u''(1)+u''(2)+u''(3))(u''(2)+u''(3)) = x^4(x+1)^2(x^4+2x^3+x^2+1).\]
Call this quantity $y$; by the definition of $\sfr_{1.54}$, we know that since $x < \sfr_{1.54}\inv$, we have $y < 1$. Also, it will be useful to note that since $x > 5/8$, we have $y > 0.9^4$.

Putting everything together, we conclude that
\[\calw^\circ_{q,3}(\ell, d, p, A) \leq 2^{d/(2k)} y^{d/2} \cdot 10 \cdot 0.9^{\abs{S_4}} \leq 2^{\ell/(2k)} y^{d/2+\abs{S_4}/4} \cdot 10.\]
We know that
\[\ell - 3 \leq \abs{S_1}+\abs{S_2}+\abs{S_3} + \abs{S_4} \leq 2d + \abs{S_4},\]
so in fact
\[\calw^\circ_{q,3}(\ell, d, p, A)\leq 10 \cdot 2^{\ell/(2k)} y^{(\ell-3)/4}.\]
After applying \cref{lem:redweight} and summing over $d,p,A$, we conclude that
\[\calw_{q, 3}(\ell) \leq 10 \cdot \ell^{10k^2+k+2} 1000^{10k^2+\ceil{\ell/k}} \cdot 2^{\ell/(2k)} y^{(\ell-3)/4}.\]
Summing over all $\ell$, this converges if $1000^{1/k} \cdot 2^{1/(2k)} \cdot y^{1/4} < 1$. Since $y < 1$, this is true for sufficiently large $k$.

\subsection{Case V: \texorpdfstring{$4 \leq q' \leq 8$}{4 ≤ q' ≤ 8}}
\begin{defn}
If $\ell$ and $q$ are positive integers and $A$ is a subset of $[\ell]$ of size $k$, let $\cals^\circ_{q}(\ell, A)$ consist of words $w_1w_2\cdots w_\ell$ such that $w_i \in [q]$ for all $i$, and $w_i + w_j \geq q$ whenever $i + j \in A \pmod{\ell}$.
\end{defn}
\begin{lem} \label{lem:q4break}
If $q' \geq 4$,
\[\calw^\circ_{q,q'}(\ell, d, p, A) \leq x^{-(q'-1) \ceil{\ell/k}}\calw^\circ_{q'}(d, A) \calw^\circ_{q'-1}(\ell+\ceil{\ell/k}-d, \ell+1-A).\]
\end{lem}
\begin{proof}
For $w \in \cals^\circ_{q,q'}(\ell, d, p, A)$, let $\Phi_1(w) = w_1w_2\cdots w_d$ and $\Phi_2(w) = w_\ell w_{\ell-1} \cdots w_{d+1} (q'-1)^{\ceil{\ell/k}}$, where $(q' - 1)^a$ denotes $q' - 1$ repeated $a$ times. We claim that $\Phi_1(w) \in \cals^\circ_{q'}(d, A)$ and $\Phi_2(w) \in \cals^\circ_{q'-1}(\ell + \ceil{\ell/k} - d, \ell + 1 - A)$.

The fact that $\Phi(w) \in \cals^\circ_{q'}(d, A)$ follows immediately from \cref{prop:simpprop}. For the second claim, we first note that \cref{prop:simpprop} implies that $\Phi_2(w)_i \leq q'-1$ for all $i$. To show the second condition, let $d' = d-\ceil{\ell/k}$, suppose $w = v^\circ$, and take $i, j \in [\ell-d']$ so that $i + j \in \ell+1-A \pmod{\ell-d'}$. We wish to show that $\Phi_2(w)_i + \Phi_2(w)_j \geq q' - 1$. If $i > \ell-d$, then $\Phi_2(w)_i = q'-1$, so the condition is automatically satisfied. We are similarly done if $j > \ell-d$, so assume $i, j \in [\ell-d]$, so that $\Phi_2(w)_i = w_{\ell+1-i}$ and $\Phi_2(w)_j = w_{\ell+1-j}$. If $w_{\ell+1-i} \neq v_{\ell+1-i}$, then $w_{\ell+1-i} = q'-1$, so the condition is true in this case. Similarly, we are done if $w_{\ell+1-j} \neq v_{\ell+1-j}$. Now, since $\max A \leq d$, we have $\min (\ell+1-A) \geq \ell+1-d$. Since $(\ell+1-d)+(\ell-d') > 2(\ell-d) \geq i+j$, it must be true that $i + j \in \ell+1-A$. Therefore $w_{\ell+1-i} + w_{\ell+1-j} = v_{\ell+1-i} + v_{\ell+1-j} \geq v_{\ell+1-i-j} - 1 = q' - 1$, as desired.

Therefore, the map $\Phi = (\Phi_1, \Phi_2)$ is an injection
\[\cals^\circ_{q,q'}(\ell, d, p, A) \to \cals^\circ_{q'}(d, A) \times \cals^\circ_{q'-1}(\ell+ \ceil{\ell/k} - d, \ell+1-A).\] Also, we have that $z(\Phi_1(w))z(\Phi_2(w)) = x^{(q'-1)\ceil{\ell/k}} z(w)$, so the lemma follows.
\end{proof}

\begin{lem} \label{lem:cleanbound}
For a given $3 \leq q \leq 8$ there is a constant $k_0$ independent of $\ell$ and $A$ such that 
\[\calw^\circ_q(\ell, A) \leq 0.99^\ell\]
whenever $k > k_0$.
\end{lem}
\begin{proof}
An application of \cref{cor:zhaocor}\footnote{We set $V 
 =[q]$, $z(v) = x^v$, $t = q$, and $d = \ell$. The parameters $k$ and $A$ are the same.} tells us that
\[\calw^\circ_q(\ell, A) \leq \paren*{\sum_{(a_1,\ldots,a_k,b_1,\ldots,b_k) \in \mathcal{X}} x^{\sum_{i=1}^k a_i + b_i}}^{\ell/(2k)},\]
where $\mathcal{X}$ is the set of tuples $(a_1,\ldots,a_k,b_1,\ldots,b_k) \in [q]$ with the condition $\min(a_1,\ldots,a_k) + \min(b_1,\ldots,b_k) \geq q$. Each element of $\mathcal{X}$ must have an integer $1 \leq a \leq q-1$ such that $a_i \geq a$ and $b_i \geq q-a$ for all $i$, so
\[\sum_{(a_1,\ldots,a_k,b_1,\ldots,b_k) \in \mathcal{X}} x^{\sum_{i=1}^k a_i + b_i} \leq \sum_{a=1}^{q-1} \paren*{\sum_{b=a}^q x^b}^k \paren*{\sum_{b=q-a}^{q} x^b}^k.\]
In particular, if we let
\[h_{q}(x) = \max_{a \in [q-1]} \paren*{\sum_{b=a}^q x^b} \paren*{\sum_{b=q-a}^{q} x^b},\]
we have
\[\sum_{a=1}^{q-1} \paren*{\sum_{b=a}^q x^b}^k \paren*{\sum_{b=q-a}^{q} x^b}^k \leq (q-1)h_q(x)^k.\]
As a result, we find that
\[\calw^\circ_q(\ell, A) \leq ((q-1) h_q(x)^k)^{\ell/(2k)} = ((q-1)^{1/(2k)} h_q(x)^{1/2})^\ell.\]
To show that this is less than $0.99^\ell$ for large $k$, we only need to show that $h_q(x) < 0.99^2$. Since $h_q(x)$ is increasing in $x$, this may be shown by verifying $h_q(\sfr_{1.54}^{-1}) < 0.99^2$ for all $3 \leq q \leq 8$.
\end{proof}

To finish, we note that by combining \cref{lem:q4break} and \cref{lem:cleanbound} we find that for sufficiently large $k$,
\[\calw^\circ_{q,q'}(\ell, d, p, A) \leq x^{-(q'-1)\ceil{\ell/k}} 0.99^{\ell+\ceil{\ell/k}} \leq 100^{\ceil{\ell/k}} 0.99^\ell,\]
so, applying \cref{lem:redweight} and summing over $d,p,A$, we get
\[\calw_{q,q'}(\ell) \leq \ell^{10k^2+k+2}1000^{10k^2 + \ceil{\ell/k}} 100^{\ceil{\ell/k}} 0.99^\ell.\]
The sum over all $\ell$ converges as long as $100000^{1/k} \cdot 0.99 < 1$, which is true for large enough $k$.

\section{Statistics of Numerical Semigroups} \label{sec:statistics}
In this section we will apply our results to questions regarding counts of numerical semigroups.
\subsection{Proof of \texorpdfstring{\cref{thm:ngasymp}}{Theorem \ref{thm:ngasymp}}}
Since $F_n = \frac{1}{\sqrt{5}}(\varphi^n - (-\varphi)^{-n})$, by applying \cref{cor:tgdecomp} we have
\begin{align*}
n_g - S\varphi^g &= \hat n_g + F_{g+1} + \sum_{i=1}^g s_i F_{g+1-i} - \frac{\varphi^{g+1}}{\sqrt{5}} - \sum_{i=1}^\infty s_i \frac{\varphi^{g+1-i}}{\sqrt{5}} \\
&= \hat n_g - \frac{(-\varphi)^{-(g+1)}}{\sqrt{5}} - \sum_{i=1}^g s_i \frac{(-\varphi)^{-(g+1-i)}}{\sqrt{5}} - \sum_{i=g+1}^\infty s_i \frac{\varphi^{g+1-i}}{\sqrt{5}}.
\end{align*}
It suffices to prove that for every $\eps > 0$, the absolute value of each of these terms is bounded above by $(\sfr_{1.54} + \eps)^g$ for large $g$.

For the first term, this follows immediately from \cref{thm:mainbound}. The second term decreases to $0$ as $g$ grows, so it is certainly bounded by $(\sfr_{1.54} + \eps)^g$. For the third, note that
\[\abs*{\sum_{i=1}^g s_i \frac{(-\varphi)^{-(g+1-i)}}{\sqrt{5}}} \leq \sum_{i=1}^g s_i.\]
By \cref{thm:mainbound}, we may assume $s_i < (\sfr_{1.54} + \eps/2)^i$ for large $i$, so there is some constant $C$ such that
\[\sum_{i=1}^g s_i < C + \sum_{i=1}^g (\sfr_{1.54} + \eps/2)^i < C + \frac{(\sfr_{1.54}+\eps/2)^{g+1}}{\sfr_{1.54}+\eps/2 - 1},\]
which is indeed less than $(\sfr_{1.54}+\eps)^g$ for large $g$. For the fourth, note that if $g$ is sufficiently large,
\[\sum_{i=g+1}^\infty s_i \frac{\varphi^{g+1-i}}{\sqrt{5}} < \sum_{i=g+1}^\infty (\sfr_{1.54} + \eps/2)^i \varphi^{g+1-i} = (\sfr_{1.54} + \eps/2)^{g+1} \sum_{j=0}^\infty \paren*{\frac{\sfr_{1.54}+\eps/2}{\varphi}}^j.\]
As long as $\sfr_{1.54}+\eps/2 < \varphi$, this last sum converges to a constant independent of $g$, so this last term is also bounded by $(\sfr_{1.54} + \eps)^g$. This concludes the proof.

\subsection{Proof of \texorpdfstring{\cref{thm:twoeps}}{Theorem \ref{thm:twoeps}}}
\begin{prop} \label{prop:twoepslem}
Suppose $w$ is a $3$-Kunz word such that $f(w) > (2+\eps)m(w)$. Then the genus of the prefix (in the sense of \cref{def:prefix}) of $w$ is at least $\frac{\eps}{3} g(w)$.
\end{prop}
\begin{proof}
Let $v$ be the prefix of $w$; call its length $\ell$. Since $f(w) > 2m(w)$, we must have $q(w)=3$, so by \cref{prop:kunzprops}, $f(w) = 2m(w) + \ell$. Thus, we find that $g(v) > \ell > \eps m(w) > \frac{\eps}{3}g(w)$.
\end{proof}

By \cref{prop:twoepslem} and mimicking the proof of \cref{cor:tgdecomp} we find that
\[n_{g,\eps} \leq \hat n_g + \sum_{i=\ceil{\frac{\eps}{3} g}}^g s_i F_{g+1-i}.\]
It can be easily proved by induction that $F_k < \varphi^k$ for all $k \geq 0$. Moreover, by \cref{thm:mainbound} we know that for sufficiently large $g$, we have $\hat n_g, s_g < 1.55^g$. Therefore, for sufficiently large $g$,
\[n_{g,\eps} < 1.55^g + \sum_{i=\ceil{\frac{\eps}{3} g}}^g \varphi^{g+1} \paren*{\frac{1.55}{\varphi}}^i < 1.55^g + \varphi^{g+1} \frac{(1.55/\varphi)^{\frac{\eps}{3} g}}{1-1.55/\varphi}.\]
Therefore
\[\limsup_{g \to \infty} n_{g,\eps}^{1/g} \leq \varphi \paren*{\frac{1.55}{\varphi}}^{\eps/3}.\]

\subsection{The statistic \texorpdfstring{$f - 2m$}{f − 2m}}
The strategy of decomposing a numerical semigroup into a stressed word and $2$-Kunz word, as in \cref{prop:3kdecomp}, is also useful in answering various questions regarding the properties of a typical numerical semigroup. For example, we have the following result:
\begin{thm} \label{thm:fm2m}
Let $\cals_\ell$ be the set of stressed words of length $\ell$ and let $\Lambda_g$ be a random numerical semigroup of genus $g$ (under the uniform distribution). Then, for any integer $k$,
\[\lim_{g \to \infty}\setp[f(\Lambda_g) - 2m(\Lambda_g) = k] = \begin{cases}
\frac{1}{\sqrt{5} \cdot S}\varphi^k & k < 0, \\
0 & k = 0, \\
\frac{1}{\sqrt{5} \cdot S} \sum_{w \in \cals_k} \varphi^{1 - g(w)} & k > 0.
\end{cases}\]
\end{thm}
Since one can show that
\[\sum_{k < 0} \frac{1}{\sqrt{5} \cdot S}\varphi^k + \sum_{k > 0} \frac{1}{\sqrt{5} \cdot S} \sum_{w \in \cals_k} \varphi^{1 - g(w)} = 1,\]
this result implies that the random variable $f(\Lambda_g) - 2m(\Lambda_g)$ has a limiting distribution that can be computed in terms of stressed words. This is a refinement of the following result by Singhal:
\begin{thm}[{Singhal \cite[Thm.\ 8]{Singhal}}]
For all $\eps > 0$, there is a positive integer $N$ such that $\lim_{g \to \infty} \setp[\abs{f(\Lambda_g)-2m(\Lambda_g)} > N] < \eps$.
\end{thm}
Singhal's result was itself a strengthening of the following result of Kaplan and Ye:
\begin{thm}[{Kaplan and Ye \cite[Thm.\ 4]{KaplanYe}}]
For all $\eps > 0$, \[\lim_{g \to \infty} \setp[2-\eps < f(\Lambda_g)/m(\Lambda_g) < 2+\eps] = 1.\]
\end{thm}
\begin{proof}[Proof of \cref{thm:fm2m}]
For each integer $k$ and $g > 0$ let $n_{g, k}$ be the number of numerical semigroups of genus $g$ satisfying $f - 2m = k$. Since for every numerical semigroup $\Lambda$ we have $m(\Lambda) \in \Lambda$ and $f(\Lambda) \notin \Lambda$, we find $n_{g,0} = 0$. Thus the theorem is true for $k = 0$. Recall the result in \cref{prop:kunzprops} that $f = (q-1)m + j$, where $j$ is maximal such that $w_j = q$. Thus, we want to count semigroups such that $j = k + (3-q)m$.

Suppose $k < 0$. Then any semigroup with $f - 2m = k$ satisfies $q < 3$. If $q \leq 1$, there is at most one semigroup of genus $g$. If $q = 2$, then we have that $j = m + k$, so we need to count the number of words $w_1w_2\cdots w_{m-1}$ such that $w_i \in [2]$, $w_{m+k} = 2$, and $w_i = 1$ for all $i > m + k$. Removing the last $-k$ entries, which are completely determined, we wish to count $2$-Kunz words of genus $g-k-1$, of which there are $F_{g+k}$ if $g \geq -k$, by \cref{prop:2kunz}. Therefore, for large $g$, we have $F_{g+k} \leq n_{g,k} \leq F_{g+k}+1$, implying that
\[\lim_{g \to \infty} \frac{n_{g,k}}{n_g} = \lim_{g \to \infty} \frac{\frac{1}{\sqrt{5}} \varphi^{g+k}}{S\varphi^g} = \frac{\varphi^k}{\sqrt{5} \cdot S}.\]

Now suppose $k > 0$. Since $j = k + (3-q)m$ and $j < m$ we must have $q \geq 3$. For a semigroup of depth $3$, the quantity $j$ is simply the length of the prefix of the associated Kunz word (in the sense of \cref{def:prefix}). Thus, in light of \cref{prop:3kdecomp}, the number of semigroups of depth $3$ with $f-2m=k$ is, for sufficiently large $g$, given by $\sum_{w \in \cals_k} F_{g+1-g(w)}$. This implies that for large $g$,
\[\sum_{w \in \cals_k} F_{g+1-g(w)} \leq n_{g,k} \leq \hat n_g + \sum_{w \in \cals_k} F_{g+1-g(w)}.\]
Therefore,
\[\lim_{g \to \infty} \frac{n_{g,k}}{n_g} = \lim_{g \to \infty} \frac{\frac{1}{\sqrt{5}} \sum_{w \in \cals_k} \varphi^{g+1-g(w)}}{S\varphi^g} = \frac{\sum_{w \in \cals_k} \varphi^{1-g(w)}}{\sqrt{5} \cdot S}.\]
This concludes the proof.
\end{proof}
This technique can likely be extended to answer related statistical questions.

\section{Computational Results, Conjectures, and Further Directions} \label{sec:program}
\subsection{Enumerating stressed words}
A stressed word of length $\ell$ can be considered as a partition of the set $[\ell-1]$ into three sets $A$, $B$, $C$, representing the locations of ones, twos, and threes, such that $\ell \notin A + A$ and $C$ is disjoint from $A+A$. Such a word would have genus $3 + \abs{A} + 2\abs{B} + 3\abs{C} = 2\ell + 1 - \abs{A} + \abs{C}$. In particular, for a given $\ell$ and $A$ with $\ell \notin A+A$, the number of stressed words of genus $2\ell + 1 - \abs{A} + k$ is $\binom{\abs{[\ell-1] \setminus (A \cup (A+A))}}{k}$.

For a given $\ell$, it is relatively straightforward to iterate through the $3^{\ceil{\ell/2}-1}$ possibilities for $A$ by choosing at most one element from each of the sets $\set{1,\ell-1}, \set{2,\ell-2}, \ldots, \set{\ceil{\ell/2}-1, \ell+1-\ceil{\ell/2}}$. By recording the number of possibilities for $A$ with a given value of $\abs{A}$ and $\abs{[\ell-1] \setminus (A \cup (A+A))}$, the stressed words of a given length can be enumerated by genus. Moreover, if $A$ is built recursively, the set $S = [\ell-1] \setminus (A \cup (A+A))$ can be memoized, so that if an element is added to $A$, the set $S$ can be updated in $O(\ell)$ time.

Using this technique, we wrote a computer program (available at \url{https://github.com/zhdag/stressed}) to enumerate all stressed words of length at most $56$ or genus at most $95$. While the latter computation required the additional consideration of the range $57 \leq \ell \leq 62$, the program only considered possible $A$ that were sufficiently large to create words of genus at most $95$, drastically reducing the search space. Another minor optimization was the observation that for the case when $k = 0$, the number $\abs{[\ell-1] \setminus (A \cup (A+A))}$ did not need to be computed; in particular, this idea led to the $\ell = 62$ case being dropped entirely. In total, these computations required approximately 6.5 hours of single-core processor time on the author's laptop.

\begin{rmk} \label{rmk:scorr}
The observations in this section allow our relation $S=\frac{\varphi}{\sqrt{5}}(1 + \sum_{g \geq 3} s_g \varphi^{-g})$ to be related to several expressions for $S$ found elsewhere in the literature. If we fix $\ell$ and $A$, the sum of $\varphi^{-g(w)}$ over all stressed words $w$ of length $\ell$ and ones in locations given by $A$ is
\begin{align*}
\varphi^{-(2\ell + 1 - \abs{A})} (1 + \varphi\inv)^{\abs{[\ell-1]\setminus (A \cup (A+A))}} &= \varphi^{-2\ell - 1 + \abs{A}+\abs{[\ell-1]\setminus (A \cup (A+A))}} \\ &= \varphi^{-\ell-2+\abs{A}-\abs{[\ell] \cap (A\cup(A+A))}}.
\end{align*}
Summing over all $A$ and $\ell$, we obtain an expression for $S$ that matches that in \cite[Thm.\ 3.11]{Zhao2010}. Another expression for $S$ is given by O'Dorney \cite[Lems.\ 10, 11]{Odorney}\footnote{Lemma 11 in \cite{Odorney} has a minor error and should instead state $C = S\varphi^2$. }, who shows that $S = \frac{\varphi}{\sqrt{5}} (1 + \sum_{k \geq 1} a_k \varphi^{-(k+2)})$, where $a_k$ is the number of strongly descended semigroups of genus $2k+1$ and efficacy $k+1$. This too can be related to the above two sums, and a brief sketch of the correspondence is as follows: associate a pair $(\ell, A)$ with a Kunz word $w_1w_2\cdots w_{\ell + 1 - \abs{A} + 2\abs{[\ell] \cap (A\cup(A+A))}}$, where $w_i = 2$ for $i \in [\ell]\setminus A$ and $w_i = 1$ for all other $i$.
\end{rmk}

\subsection{Results}
\begin{table}[tbp]\centering { \footnotesize
\begin{tabular}{cccccccccccc}  \toprule
$\ell$ & Sum & $\ell$ & Sum & $\ell$ & Sum & $\ell$ & Sum & $\ell$ & Sum & $\ell$ & Sum \\ \midrule
1 & 0.236068 & 11 & 2.067270 & 21 & 3.288559 & 31 & 3.903892 & 41 & 4.167027 & 51 & 4.269349 \\
2 & 0.381966 & 12 & 2.182360 & 22 & 3.355010 & 32 & 3.933559 & 42 & 4.178757 & 52 & 4.273803 \\
3 & 0.618034 & 13 & 2.368955 & 23 & 3.454253 & 33 & 3.976294 & 43 & 4.196208 & 53 & 4.280242 \\
4 & 0.763932 & 14 & 2.478676 & 24 & 3.509096 & 34 & 4.001513 & 44 & 4.206072 & 54 & 4.283812 \\
5 & 1.005025 & 15 & 2.640734 & 25 & 3.593944 & 35 & 4.038160 & 45 & 4.220026 & 55 & 4.289037 \\
6 & 1.145898 & 16 & 2.737483 & 26 & 3.643206 & 36 & 4.058413 & 46 & 4.228224 & 56 & 4.291988 \\
7 & 1.380047 & 17 & 2.886712 & 27 & 3.713999 & 37 & 4.089100 & 47 & 4.240024 & & \\
8 & 1.517814 & 18 & 2.970980 & 28 & 3.755414 & 38 & 4.106684 & 48 & 4.246500 & & \\
9 & 1.731797 & 19 & 3.102391 & 29 & 3.817255 & 39 & 4.131525 & 49 & 4.256156 & & \\
10 & 1.862340 & 20 & 3.177471 & 30 & 3.851635 & 40 & 4.145897 & 50 & 4.261649 & & \\ \bottomrule
\end{tabular}}
\caption{Partial sums of $ \sum_{w\text{ stressed}} \varphi^{-g(w)}$ considering all stressed words of length at most $\ell$. This table is largely equivalent to Table 2 in \cite{Zhao2010}, by \cref{rmk:scorr}. All results with $\ell > 46$ are new.} \label{tab:smalltable}
\end{table}
\begin{table}[tbp] \centering {\tiny
\begin{tabular}{crrr@{\hspace{1cm}}crrr}  \toprule
\footnotesize $g$ & \footnotesize $s_g$ & \footnotesize $t_g$ & \footnotesize $\hat n_g$ & \footnotesize $g$ & \footnotesize $s_g$ & \footnotesize $t_g$ & \footnotesize $\hat n_g$ \\ \midrule
0 & 0 & 1 & 0 & 48 & 398937594 & 35227607540 & 3032888834 \\
1 & 0 & 1 & 0 & 49 & 620308837 & 57443335681 & 4756701071 \\
2 & 0 & 2 & 0 & 50 & 964299016 & 93635242237 & 7455057891 \\
3 & 1 & 4 & 0 & 51 & 1498722966 & 152577300884 & 11675899900 \\
4 & 0 & 6 & 1 & 52 & 2328886172 & 248541429293 & 18273725810 \\
5 & 1 & 11 & 1 & 53 & 3618215600 & 404736945777 & 28580512964 \\
6 & 3 & 20 & 3 & 54 & 5619924806 & 658898299876 & 44671692245 \\
7 & 2 & 33 & 6 & 55 & 8725957048 & 1072361202701 & 69779534158 \\
8 & 4 & 57 & 10 & 56 & 13542732051 & 1744802234628 & 108935597479 \\
9 & 9 & 99 & 19 & 57 & 21008277551 & 2838171714880 & 169969266940 \\
10 & 12 & 168 & 36 & 58 & 32573278946 & 4615547228454 & 265059561556 \\
11 & 20 & 287 & 56 & 59 & 50480678072 & 7504199621406 & 413144466289 \\
12 & 32 & 487 & 105 & 60 & 78197851828 & 12197944701688 & 643658549663 \\
13 & 50 & 824 & 177 & 61 & 121086932116 & 19823231255210 & 1002326746843 \\
14 & 84 & 1395 & 298 & 62 & 187445618110 & 32208621575008 & 1560141961678 \\
15 & 132 & 2351 & 506 & 63 & 290118087627 & 52321970917845 & 2427273997885 \\
16 & 208 & 3954 & 852 & 64 & 448979969989 & 84979572462842 & 3774618610486 \\
17 & 331 & 6636 & 1409 & 65 & 694763898132 & 137996307278819 & 5867177646731 \\
18 & 526 & 11116 & 2351 & 66 & 1074945125010 & 224050824866671 & 9115752259043 \\
19 & 841 & 18593 & 3871 & 67 & 1662803652299 & 363709935797789 & 14156971708484 \\
20 & 1333 & 31042 & 6354 & 68 & 2571392093291 & 590332152757751 & 21977155500049 \\
21 & 2145 & 51780 & 10414 & 69 & 3975074031374 & 958017162586914 & 34103955827937 \\
22 & 3401 & 86223 & 17023 & 70 & 6142744228748 & 1554492059573413 & 52902754596745 \\
23 & 5314 & 143317 & 27646 & 71 & 9489097657132 & 2521998319817459 & 82034862865123 \\
24 & 8396 & 237936 & 44892 & 72 & 14653792397521 & 4091144171788393 & 127165544752421 \\
25 & 13279 & 394532 & 72692 & 73 & 22623678356496 & 6635766169962348 & \thead{c}{---} \\
26 & 20952 & 653420 & 117412 & 74 & 34921396940988 & 10761831738691729 & \thead{c}{---} \\
27 & 33029 & 1080981 & 189286 & 75 & 53896060190628 & 17451493968844705 & \thead{c}{---} \\
28 & 51927 & 1786328 & 304702 & 76 & 83170721269779 & 28296496428806213 & \thead{c}{---} \\
29 & 81527 & 2948836 & 489003 & 77 & 128330590081463 & 45876320987732381 & \thead{c}{---} \\
30 & 128102 & 4863266 & 783507 & 78 & 197980950188515 & 74370798366727109 & \thead{c}{---} \\
31 & 201700 & 8013802 & 1252986 & 79 & 305373013675616 & 120552492368135106 & \thead{c}{---} \\
32 & 317461 & 13194529 & 2000541 & 80 & 470903017814500 & 195394193752676715 & \thead{c}{---} \\
33 & 498911 & 21707242 & 3188964 & 81 & 725955935033818 & 316672642055845639 & \thead{c}{---} \\
34 & 782868 & 35684639 & 5076448 & 82 & 1118822820960919 & 513185658629483273 & \thead{c}{---} \\
35 & 1226255 & 58618136 & 8069065 & 83 & 1723818676661774 & 831582119361990686 & \thead{c}{---} \\
36 & 1919070 & 96221845 & 12810655 & 84 & 2655337007585162 & 1347423114999059121 & \thead{c}{---} \\
37 & 3000905 & 157840886 & 20317403 & 85 & 4089504559715309 & 2183094738920765116 & \thead{c}{---} \\
38 & 4687213 & 258749944 & 32189863 & 86 & 6297510121544894 & 3536815364041369131 & \thead{c}{---} \\
39 & 7315975 & 423906805 & 50944640 & 87 & 9696835494675713 & 5729606938456809960 & \thead{c}{---} \\
40 & 11419861 & 694076610 & 80537674 & 88 & 14929924846668923 & 9281352227344848014 & \thead{c}{---} \\
41 & 17833383 & 1135816798 & 127176042 & 89 & 22984927653735759 & 15033944093455393733 & \thead{c}{---} \\
42 & 27857264 & 1857750672 & 200605850 & 90 & 35381018243394036 & 24350677339043635783 & \thead{c}{---} \\
43 & 43511423 & 3037078893 & 316112953 & 91 & 54452928150272573 & 39439074360649302089 & \thead{c}{---} \\
44 & 67908811 & 4962738376 & 497663200 & 92 & 83788019525979642 & 63873539719218917514 & \thead{c}{---} \\
45 & 105857661 & 8105674930 & 782811886 & 93 & 128897984242944822 & 103441512064111164425 & \thead{c}{---} \\
46 & 164837336 & 13233250642 & 1230383006 & 94 & 198251405524002501 & 167513303188854084440 & \thead{c}{---} \\
47 & 256493732 & 21595419304 & 1932426198 & 95 & 304861061425644202 & 271259676314390893067 & \thead{c}{---} \\ \bottomrule
\end{tabular}}
\caption{Known values for $s_g$, $t_g$, and $\hat n_g$. All results with $g > 65$ are new.} \label{tab:bigtable}
\end{table}
Our first result is a new lower bound of the quantity $S$. \Cref{tab:smalltable} shows the partial sums of the sum  $\sum_{w\text{ stressed}} \varphi^{-g(w)}$, considering all stressed words of length at most $\ell$. If we use the $\ell = 56$ partial sum as a lower bound, we obtain the following bound:
\begin{prop}
$S > 3.8293$.
\end{prop}
A crude extrapolation suggests that the true value of $S$ lies between $3.85$ and $3.86$.

We will note, due to \cref{rmk:scorr}, that the contents of \cref{tab:smalltable} are largely equivalent to Table 2 in \cite{Zhao2010}, which computes this sum up to $\ell = 46$, concluding that $S > 3.78$. As such, our computations yield $10$ additional terms.

A larger improvement on existing knowledge is obtained by the calculation of $s_g$ for all $g \leq 95$ (see \cref{tab:bigtable}). Using the fact that $s_g = t_g - t_{g-1} - t_{g-2}$, we are then able to deduce the value of $t_g$ for all $g \leq 95$. To our best knowledge, the highest $g$ for which $t_g$ had been previously calculated is $g = 65$ \cite{Eliahou2020}, so this computation adds $30$ previously unknown terms to the sequence.

Combining these results for $t_g$ with known values of $n_g$ for $g \leq 72$ \cite[\href{https://oeis.org/A007323}{A007323}]{BrasAmoros2020,oeis} yields values for $\hat n_g$ for $g \leq 72$. While not attempted in this paper, it may be possible to construct algorithms for evaluating $\hat n_g$ that are significantly faster than algorithms for computing $n_g$, considering that $\lim_{g \to \infty} \hat n_g/n_g = 0$. If this is possible, then new values for $\hat n_g$ may be combined with values of $t_g$ computed in this paper to yield further terms of the sequence $n_g$.

\subsection{Conjectural asymptotics}
\begin{figure}[tbp]\centering
\includegraphics{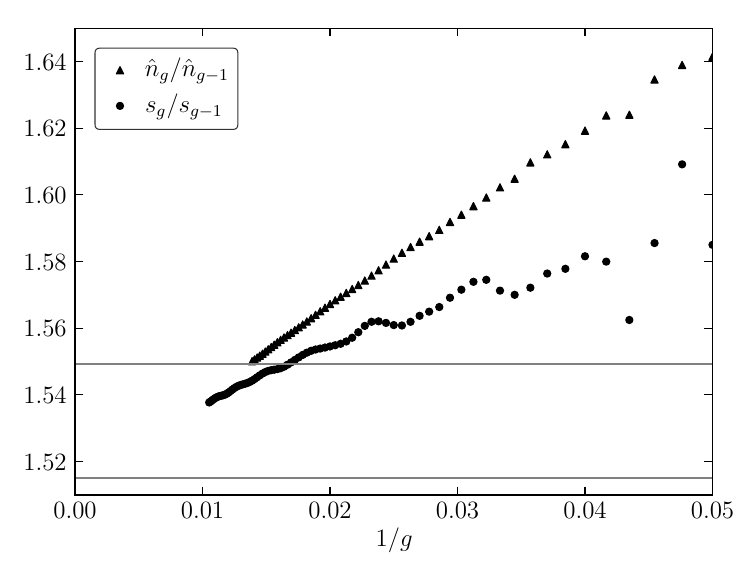}
\caption{A plot of the ratios $\hat n_g/\hat n_{g-1}$ and $s_g/s_{g-1}$ against $1/g$. The horizontal lines are located at $\sfr_{1.51}$ and $\sfr_{1.54}$.} \label{fig:1}
\end{figure}
\Cref{thm:mainbound} bounds the growth rates $r_1$ and $r_2$ within an interval with endpoints $\sfr_{1.51}$ and $\sfr_{1.54}$, but does not determine them exactly. To analyze the numerical evidence, in \cref{fig:1} we plot the ratios $\hat n_g/\hat n_{g-1}$ and $s_g/s_{g-1}$ against $1/g$. We use $1/g$ instead of the more traditional $g$ in order to better visualize the process $g \to \infty$. A extrapolation of the data suggests not only that $r_1 = r_2 = \sfr_{1.51}$, but that $\lim_{g \to \infty} \hat n_g/\hat n_{g-1} = \lim_{g \to \infty} s_g/s_{g-1} = \sfr_{1.51}$. As mentioned in \cref{sec:upper}, all that would be needed to show that $r_1 = r_2 = \sfr_{1.51}$ would be an improvement of the arguments in \cref{subsec:caseiv}, whereas proving $\lim_{g \to \infty} \hat n_g/\hat n_{g-1} = \lim_{g \to \infty} s_g/s_{g-1} = \sfr_{1.51}$ would require additional techniques beyond analyzing the convergence of generating functions.

On the other hand, it appears unlikely that $s_g$ and $\hat n_g$ are asymptotic to exponential functions, since the ratios $\hat n_g/\hat n_{g-1}$ and $s_g/s_{g-1}$ appear to have a nonzero slope in terms of $1/g$; in other words, there appear to be constants $\alpha$ and $\alpha'$ (possibly equal) such that $\hat n_g/\hat n_{g-1} = \sfr_{1.51}(1 + \alpha/g + o(1/g))$ and $s_g/s_{g-1} = \sfr_{1.51}(1 + \alpha'/g + o(1/g))$. This suggests that $\hat n_g/\sfr_{1.51}^g$ and $s_g/\sfr_{1.51}^g$ grow polynomially in $g$ with exponents $\alpha$ and $\alpha'$. In summary, we may condense our above discussion into the following three successively stronger conjectures, where each implies the previous:
\begin{conj} \label{conj:growthconj}
\begin{parts}
\item $\lim_{g \to \infty} \hat n_g^{1/g} = \lim_{g\to\infty} s_g^{1/g} = \sfr_{1.51}$.
\item $\lim_{g \to \infty} \hat n_g/\hat n_{g-1} = \lim_{g \to \infty} s_g/s_{g-1} = \sfr_{1.51}$.
\item There exist constants $\alpha$ and $\alpha'$ (possibly equal) such that $\hat n_g = g^{\alpha + o(1)} \sfr_{1.51}^g$ and $s_g = g^{\alpha' + o(1)} \sfr_{1.51}^g$.
\end{parts}
\end{conj}
Computational evidence suggests that if $\alpha$ and $\alpha'$ exist, then they are both between $1.6$ and $1.7$.

\subsection{Remarks on \texorpdfstring{\cref{conj:bafib}}{Conjecture \ref{conj:bafib}}}
\begin{figure}[tbp] \centering
\includegraphics[width=5in]{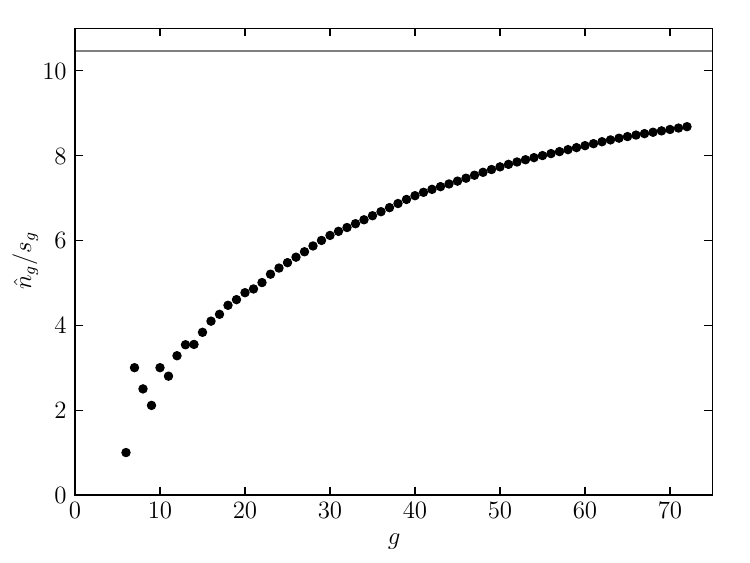}
\caption{A plot of $\hat n_g/s_g$. The horizontal line is located at $(\sfr_{1.51}^{-1} + \sfr_{1.51}^{-2}-1)\inv \approx 10.465$.} \label{fig:2}
\end{figure}
We conclude this paper by discussing \cref{conj:bafib} in light of our results, which may be written as the statement that $s_g + \hat n_g - \hat n_{g-1} - \hat n_{g-2} \geq 0$. While our results bound $\abs{s_g + \hat n_g - \hat n_{g-1} - \hat n_{g-2}} < (\sfr_{1.54} + o(1))^g$, they do not determine its sign. Nevertheless, this bound indicates that the validity of the conjecture is in a sense determined by phenomena that lie ``underneath'' the prevailing Fibonacci behavior of $n_g$.

Our work with $s_g$ and $\hat n_g$ also hints at a possible method for proving or disproving \cref{conj:bafib}. The first step lies in proving part (b) of \cref{conj:growthconj}, which will imply that $\hat n_{g-1} + \hat n_{g-2} - \hat n_g \sim (\sfr_{1.51}^{-1} + \sfr_{1.51}^{-2} - 1) \hat n_g$. Assuming this result, the truth of \cref{conj:bafib} will then be highly dependent on the ratio $\hat n_g/s_g$, which seems to approach some real constant $A$ as $g$ goes to infinity (see \cref{fig:2}). In particular, assuming that $A$ exists, we then find that
\[s_g + \hat n_g - \hat n_{g-1} - \hat n_{g-2} = (A\inv - (\sfr_{1.51}^{-1} + \sfr_{1.51}^{-2} - 1) + o(1)) \hat n_g.\]
Thus, if $A > (\sfr_{1.51}^{-1} + \sfr_{1.51}^{-2} - 1)\inv \approx 10.465$, \cref{conj:bafib} is false for sufficiently large $g$, and if $A < (\sfr_{1.51}^{-1} + \sfr_{1.51}^{-2} - 1)\inv$, then \cref{conj:bafib} is true for sufficiently large $g$. The most intriguing possibility would be if we had $A = (\sfr_{1.51}^{-1} + \sfr_{1.51}^{-2} - 1)\inv$, which would give no information as to whether \cref{conj:bafib} is true. If this were the case, it would strongly suggest a deep connection between the numbers $\hat n_g$ and $\hat n_{g-1} + \hat n_{g-2}$ in spite of the fact that the growth rate of $\hat n_g$ is less than $\varphi$. Such a connection could take the form of a novel operation on numerical semigroups.

\section*{Acknowledgments}
This work was conducted in large part at the REU at the University of Minnesota Duluth, funded by NSF grant DMS-1949884 and NSA grant H98230-20-1-0009. The author would like to thank Joseph Gallian for organizing the REU and providing frequent feedback and support. The author would also like to thank Amanda Burcroff, Mehtaab Sawhney, Cynthia Stoner, Yufei Zhao, and the anonymous referees for helpful discussions and comments on the manuscript.

\printbibliography
\end{document}